%% file: main.tex
\documentclass[11pt,twoside,reqno]{amsart}
\usepackage[all]{xy}
        \usepackage {amssymb,latexsym,amsthm,amsmath,mathtools,dsfont,mathrsfs, blkarray, bigstrut, gauss}
        \usepackage{enumitem,color}

\usepackage{mathtools}

\usepackage{babel}
\usepackage{hyperref}
\makeatletter
\def\imod#1{\allowbreak\mkern10mu({\operator@font mod}\,\,#1)}
\makeatother
\makeatletter
\renewcommand*\env@matrix[1][*\c@MaxMatrixCols c]{%
  \hskip -\arraycolsep
  \let\@ifnextchar\new@ifnextchar
  \array{#1}}
\makeatother
\usepackage[capitalize,nameinlink]{cleveref} 
\usepackage{comment} 
\usepackage{xcolor} 
\hypersetup{
    colorlinks,
    linkcolor={red!80!black},
    citecolor={green!80!black},
    urlcolor={blue!80!black}
}
\theoremstyle{plain}
\newtheorem{theorem}{Theorem}[section]
\newtheorem{corollary}[theorem]{Corollary}
\newtheorem{lemma}[theorem]{Lemma}
\newtheorem{remark}[theorem]{Remark}
\newtheorem{proposition}[theorem]{Proposition}

\newtheorem{definition}[theorem]{Definition}
\newtheorem{example}[theorem]{Example}

\theoremstyle{definition}

\newcommand{\dlabel}[1]{\ifmmode \text{\ttfamily \upshape [#1] } \else
	{\ttfamily \upshape [#1] }\fi \label{#1}}

\newcommand{\cyc}{\mathscr{C}}
\newcommand{\pa}{\mathscr{P}}

\newcommand{\M}{\operatorname{\mathscr{M}} }

\newcommand{\Sym}{\mathfrak{S}} 

\newcommand{\bigxone}{\mbox{\normalfont\Large\bfseries $X_1$}} 
\newcommand{\bigxtwo}{\mbox{\normalfont\Large\bfseries $X_2$}} 
\newcommand{\bigxthree}{\mbox{\normalfont\Large\bfseries $X_3$}}
\newcommand{\biga}{\mbox{\normalfont\Large\bfseries $\alpha$}}

\newcommand{\cent}{\mathcal{Z}} 
\newcommand{\orb}{\mathcal{O}} 

\newcommand{\Sn}{\operatorname{Sym} }

\newcommand{\Aut}{\operatorname{Aut} } 

\title[Cycle Matrices for SYBE]{Cycle matrices: A combinatorial approach to the solutions of Quantum Yang-Baxter Equations}
\author[A. Kanrar]{Arpan Kanrar}
\email{arpankanrar000@gmail.com, arpankanrar@hri.res.in}
\author[S. Panja]{Saikat Panja}
\email{panjasaikat300@gmail.com, saikatpanja@hri.res.in}
\address{Harish-Chandra Research Institute- Main Building, Chhatnag Rd, Jhusi, Uttar Pradesh 211019, India}
\thanks{Kanrar A. is supported by SRF-PhD fellowship from HRI. Panja S. acknowledges the support of HRI through PDF-M fellowship}
\date{\today}
\subjclass[2020]{}

\begin{document}
\setlength{\baselineskip}{15pt}
\maketitle
\begin{abstract}
    An $n\times n$ matrix $M=[m_{ij}]$ with $m_{ij}\in U_n=\{1,2,\ldots,n\}$ will be called a cycle matrix if $(U_n,\cdot)$ is a cycle set, where $i\cdot j=m_{ij}$. We study these matrices in this article. Using these matrices, we give some recipes to construct solutions, which include the multipermutation level $2$ solutions. As an application of these, we construct a multipermutation solution of level $r$ for all $r\geq 1$. Our method gives alternate proof that the class of permutation groups of solutions contains all finite abelian groups.
\end{abstract}
\input{sec_introduction}

\input{sec_gen_results}

\input{sec_Sn_action}

\input{sec_construction}
\input{sec_multiperm}
\bibliographystyle{siam}

\end{document}

%% file: sec_introduction.tex
\section{Introduction}\label{sec: intro}
V. G. Drinfeld suggested the study of the set-theoretic solution to the quantum Yang-Baxter equation (QYBE) (see \cite{Dr90}).
A \emph{set-theoretic solution to the QYBE} is an ordered pair $(X,R)$, where $X$ is a set and $R:X\times X\longrightarrow X\times X$  is a map, satisfying
\begin{align*}
    R^{12}R^{13}R^{23}=R^{23}R^{13}R^{12},
\end{align*}
with $R^{ij}:X\times X\times X\longrightarrow X\times X\times X$ being the map acting on the $(i,j)$-th position by $R$.
This has been an important research topic for the last two decades after the breakthrough \cite{EtScSo99} by P. Etingof, T. Schedler and A. Soloviev.
Although such solutions were constructed by A. D. Weinstein and P. Xu (see \cite{WeXu}) and by J. H. Lu, M. Yan and Y. C. Zhu (see \cite{LuYaZh00}) independently, 
the paper by Etingof \emph{et. al.} further studies solutions with additional conditions of \emph{nondegeneracy} of $R$ (i.e. if $R(x,y)=(g_x(y),f_y(x))$ then $g_x,f_y$ are bijective for all $x,y\in X$) and $R$ being involutive (i.e. $R^2=\text{Id}_{X\times X}$) and derive powerful consequences. In this current article, by a \emph{solution}, we will mean finite non-degenerate involutive set-theoretic solution to the QYBE (in short, SYBE). Before proceeding further we will mention a few definitions. The \emph{permutation group of a solution} is the group defined as
\begin{align*}
    \mathcal{G}=\langle f_x:x\in X\rangle.
\end{align*}
A solution $(X,R)$ is said to be \emph{decomposable} if there exists disjoint subsets $X_1,X_2$ of $X$ such that $R(X_i\times X_i)\subseteq X_i\times X_i$, $(X_i,R\vert _{X_i\times X_i})$ is a solution and $X=X_1\bigsqcup X_2$. 
The solution will be called \emph{indecomposable} in case of the non-existence of such a pair $X_1,~X_2$ (it is equivalent that the natural action of the permutation group $\mathcal{G}$ on $X$ is transitive \cite{EtScSo99}).
In the paper \cite{EtScSo99}, the authors have shown that there exists a unique (up to
isomorphism) indecomposable solution order $p$, where $p$ is
a prime (see section $2.5$ and $2.6$ therein). After that several attempts have been made to look after and analyze solutions in case $|X|$ is a composite number.
For example, W. Rump has studied the decomposability of square-free solutions in \cite{Ru05}, S. Ram\'irez and L. Vendramin have studied the decomposability of solutions in \cite{RaVe22}.
Recently in an article Agore, Chirvasitu and, Militaru have proved many counting results for solutions of Frobenius-Separability (FS) type using the category of pointed Kimura semigroups (see \cite{AgChMi23}). 
Our goal here is to introduce counting methods for SYBEs using algebraic tools. Furthermore, our results give glimpses of the structure of solutions in the case of isomorphic classes of decomposable solutions, in terms of these matrices.

The concept of cycle sets was introduced by W. Rump in \cite{Ru05} and has been a source of SYBEs (see \cite{Jes16}, \cite{Rum22} and the references therein). A \emph{cycle set} is a tuple
$(X,\cdot)$ such that 
the map $\psi_x:y\mapsto x\cdot y$ is invertible, and
\begin{align*}
    (x\cdot y)\cdot (x\cdot z)=(y\cdot x)\cdot(y\cdot z),
\end{align*}
for all $x,y,z\in X$. A cycle set will be called \emph{non-degenerate} if the map $\varphi:x\mapsto x\cdot x$ is bijective. Finally, call a cycle set to be 
\emph{square-free} if $\varphi$ is the identity map. 
An important result about cycle sets states that there is a bijection between non-degenerate cycle sets and set of all SYBEs (\cite[Proposition $1$]{Ru05}).
For two solutions $(X,\cdot)$ and $(Y,\bullet)$, a \emph{homomorphism} $f$ from $(X,\cdot)$ to $(Y,\bullet)$ is a set theoretical map $f:X\longrightarrow Y$ such that $f(x\cdot y)=f(x)\bullet f(y)$ for all $x\in X$ and $y\in Y$. A bijective homomorphism from $(X,\cdot)$ to itself will be called an \emph{automorphism}. The set of all automorphisms of $(X,\cdot)$ will be
denoted by $\Aut(X,\cdot)$. \emph{We further restrain ourselves to solutions coming from cycle sets.}
 A solution $(X,\cdot)$ is called a \emph{permutation solution} if, for any $x\in X,$ $\psi_x=\sigma$ for some permutation $\sigma$ and,
it will be called a \emph{trivial solution} if $\sigma=id$. We will denote the permutation solution corresponding to $\sigma$ by $(X,\cdot_\sigma)$. 
A solution is \emph{irretractable} if the natural map $x\mapsto \psi_x$ is injective, otherwise it will be called
\emph{retractable}. The relation, $x\sim y$ if and only if $\psi_x=\psi_y$, is an equivalence relation and, the equivalence class of $x$ will be denoted by $\bar{x}$, and the set of equivalence classes will be denoted by $\bar{X}$ or $Ret(X)$. 
In \cite[pp. 157]{Ru07}, it is shown that $\Bar{x}\cdot\Bar{y}:=\overline{x\cdot y}$ defines a cycle set on $\bar{X}$.
A retractable solution is a \emph{multipermutation of level $n$}, if $n$ is the least positive integer such that $|Ret^n(X)|=1.$
\subsection*{Notations} We set some notations here. The group of all bijections of $n$ indeterminates $x_1,x_2,\cdots,x_n$ will be denoted by $\Sym_{x_1,x_2,\ldots,x_n}$. In case $x_i=i$, this group will be identified with the symmetric group on $n$ letters and will be denoted by $\Sym_n$. By $\Sn(i_1,i_2,\ldots,i_k)$ we mean the set of all bijections of $\{i_1,i_2,\ldots,i_k\}$. 
For a group $G$ and a $G$-space $Y$, the orbit of an element $y\in Y$ will be denoted as $\orb_G(y)$. The set $\{1,2,\ldots, m\}$ will be denoted as $U_m$.
The centralizer of an element $g\in G$ will be denoted by $\cent_{G}(g)$.
The number of partitions of a positive number $n$ will be denoted by $\mathscr{P}(n)$. A row of a matrix will be written in a square bracket $[~]$ and 
for a column, we will use the notation $[~]^t$.
 Other notations are standard.
\subsection*{Organization of the paper} The paper is organized as follows; In \cref{sec:general-results} we define the main object of the study, the cycle matrix.
We observe some properties of the cycle matrix and prove that the cycle matrices which give rise to decomposable solutions are singular. 
We have shown that the product of solutions has correspondence with the tensor product of cycle matrices. We define a transpose cycle matrix, which corresponds to a special kind of irretractable solution. 
We construct a collection of such cycle matrices. In \cref{sec:sn-action} we define an action of the symmetric group $\Sym_n$ on the set of all $n\times n$ cycle matrices.
Using this action we prove that the number of permutation solutions of order $n$ (up to isomorphism) is the number of partitions of $n$ and $\Aut(U_n,\cdot_\sigma)$ is the centralizer of $\sigma$ in $\Sym_n.$ \cref{sec:construction-old-new} is devoted to the construction of new solutions. 
Firstly we start with two trivial solutions $(U_n,\cdot_{id}),~(U_m,\cdot_{id})$ and construct a collection of solutions of multipermutation level $2$ on $U_{n+m}$ with respect to any partition of $U_n$. 
As a consequence of the construction we obtain the well-known result \cite{CeJeDe10} that ``all finite  abelian groups are permutation groups". In this section, we further demonstrate the construction of different solutions on $U_{\sum m_i}$, from a given finite collection of solutions $\{(U_{m_i},\cdot_i)\vert i\in I\}$. Lastly, \cref{sec:multiperm-solutions}, is devoted to the construction of a multipermutation solution of level $n$, for $n\ge 1.$

%% file: sec_gen_results.tex
\section{Results in generalities}\label{sec:general-results}
We start with the definition of a cycle matrix. This will be followed by a few examples and properties of such a matrix.
\begin{definition}[Cycle matrix]
    An $n\times n$ matrix $M=(m_{ij})$ with all entries from $\{1,2,\ldots,n\}$ is said to be a \emph{cycle matrix} if $(U_n,\cdot)$ is a non-degenerate cycle set and, $i\cdot j= m_{ij}$. Two such matrices $(m_{ij}), (m^\prime_{ij})$ will be called \emph{isomorphic} if the corresponding cycle sets are isomorphic to each other. The set of all $n\times n$ cycle matrices will be denoted by $\cyc_n$. Further, a cycle matrix will be called \emph{indecomposable}, if the corresponding SYBE (coming from the cycle set of the matrix) is indecomposable.
\end{definition}
\begin{example}
Considering $U_4=\{1,2,3,4\}$ the following two matrices 
\begin{align*}
    \begin{bmatrix}
    1 & 2 & 3 & 4\\
    1 & 2 & 3 & 4\\
    1 & 2 & 3 & 4\\
    1 & 2 & 3 & 4
    \end{bmatrix},
    \begin{bmatrix}
    1 & 2 & 3 & 4\\
    1 & 2 & 3 & 4\\
    1 & 2 & 3 & 4\\
    1 & 3 & 2 & 4
    \end{bmatrix}
\end{align*}
are two non-isomorphic cycle matrices.
\end{example}
\begin{lemma}\label{lem:existence-of-matrices}
	For any permutation $\sigma\in\Sym_n$, there is a cycle matrix whose diagonal is $\sigma$. In particular $\cyc_n\neq\emptyset$.
\end{lemma}
\begin{proof}
    Consider the permutation solution corresponding to $\sigma$. Note that in that case, the cycle matrix will be
    \begin{align*}
        \begin{bmatrix}
            \sigma(1) & \sigma(2)&  \cdots& \sigma(n)\\
            \sigma(1) & \sigma(2)&  \cdots& \sigma(n)\\ 
            \vdots & \vdots&  \ddots & \vdots\\
            \sigma(1) &  \sigma(2)&\cdots& \sigma(n)
        \end{bmatrix}.
    \end{align*}
    This finishes the proof.
\end{proof}
\begin{lemma}\label{lem:asymmetric}
    Cycle matrices are not symmetric.
\end{lemma}
\begin{proof}
Consider the cycle set $(X,\cdot)$ arising from the cycle matrix. Note that for $i\neq j$, by the non-degeneracy we have that $i\cdot i\neq j\cdot j$ for all $i\neq j$. If possible let us assume $m_{ij}=m_{ji}$ for some $i\neq j$. Then we have that
\begin{align*}
    (i\cdot j)\cdot (i\cdot i)&=(j\cdot i)\cdot(j\cdot i)\\
    &=(i\cdot j)\cdot (j\cdot j),
\end{align*}
which is not possible, since $i\cdot i\neq j\cdot j$.
\end{proof}
Note that a cycle matrix $(a_{ij})$ is not only non-symmetric but also $a_{ij}\neq a_{ji}$ for all $i,~j.$
\begin{lemma}\label{lem:intranstive-zero-determinant}
	Let $R$ be a commutative ring with unity. If $G\le \Sym_{x_1,x_2,\ldots,x_n}$ acts intransitively on $\{x_1,x_2,\dots,x_n\}$ by the natural action, then the determinant of the $n\times n$ matrix in which each row is an element of $G$, is zero in $R[x_1,x_2,\dots,x_n]$. 
\end{lemma}
\begin{proof}
    Let $M$ be an $n\times n$ matrix in which each row is an element of $G$. The subgroup $H$ of $G$ generated by all the rows, acts intransitively on $X$ by the natural action. For an orbit $\orb_H(x_1)$, there is an element $x_r\in X$ such that $x_r\not\in \orb_H(x_1)$. Let $\mathcal{C}_1$ be the set of all columns such that $\bigcup\limits_{C\in \mathcal{C}_1}C=\orb_H(x_1)$. Similarly, we have $\mathcal{C}_r$. The fact $\orb_H(x_1)\cap \orb_H(x_r)=\emptyset$, implies $\mathcal{C}_1\cap\mathcal{C}_2=\emptyset$. Choose arbitrary elements $C_1\in \mathcal{C}_1$, $C_r\in\mathcal{C}_r$ and perform two elementary column operations on $M$, that are replacement of $C_1$ by $\sum\limits_{C\in\mathcal{C}_1}C$ and $C_r$ by $\sum\limits_{C\in\mathcal{C}_r}C$. The new $C_1$ and $C_r$ columns will be $\sum\limits_{x\in \orb_H(x_1)}x[1,1,\cdots,1]^t$ and $\sum\limits_{x\in \orb_H(x_r)}x[1,1,\cdots,1]^t$, as each row is a permutation. Hence determinant of the matrix $M$ in $R[x_1,x_2,\dots,x_n]$ is $0$. 
\end{proof}
\begin{corollary}\label{cor:non-zero-indecomposable}
    A cycle matrix with a non-zero determinant gives an indecomposable solution.
\end{corollary}
\begin{proof}
    Setting $x_i=i\in\mathbb{Z}$ and $R=\mathbb{Z}$ in \cref{lem:intranstive-zero-determinant}, we get that if the cycle matrix gives a decomposable solution, then the determinant must be zero. Hence the result follows by taking the contrapositive. 
\end{proof}
The converse of the \cref{cor:non-zero-indecomposable} is not true, as evident from the following example.
\begin{example}
We note down the following two matrices. The first one is an indecomposable cycle matrix, although it is of determinant $0$. The second one gives an indecomposable cycle matrix, illustrating the previous corollary.
 \begin{align*}
     \begin{vmatrix}
		4&7&2&1&6&5&8&3\\
		8&1&4&3&2&7&6&5\\
		4&3&2&5&6&1&8&7\\
		2&1&6&3&8&7&4&5\\
		2&5&4&3&8&7&6&1\\
		4&3&8&1&6&5&2&7\\
		2&1&4&7&8&3&6&5\\
		6&3&2&1&4&5&8&7
		\end{vmatrix}=0, 
  \begin{vmatrix}
	4&8&7&1&5&6&3&2\\
	7&1&4&8&3&2&5&6\\
	5&3&2&6&1&4&7&8\\
	2&6&5&3&7&8&1&4\\
	3&6&5&2&8&7&1&4\\
	1&8&7&4&6&5&3&2\\
	7&4&1&8&3&2&6&5\\
	5&2&3&6&1&4&8&7
	\end{vmatrix}=-147456.
 \end{align*}
\end{example}
The following lemma is well known, but we mention it here for completeness.
\begin{lemma}\label{lem:product-cycle-set}
    Let $(X,\cdot)$ and $(Y,\star)$ are two cycle sets, then $(X\times Y,\circ)$ is a cycle set, where $(x_1,y_1)\circ(x_2,y_2)=(x_1.x_2,y_1\star y_2).$ 
\end{lemma}
The cycle set $(X\times Y,\circ)$ is said to be the product of cycle sets $(X,.)$ and $(X,\star)$.
We connect this with the concept of the tensor product of matrices. Recall that, for two matrices $A=(a_{ij})$ and $B=(b_{kl})$, the tensor matrix is a block matrix such that the $ij$-th block is given by
\begin{align*}
    \left(A\otimes B\right)_{ij}=\left(a_{ij}B\right).
\end{align*}
\begin{proposition}\label{prop:tensor-multiplication}
Consider two cycle matrices $A\in\cyc_m$ and $B\in\cyc_n$. If the corresponding cycle sets are given by $(U_m,\cdot)$ and $(U_n,\star)$ respectively, then the matrix $A\otimes B$ is a cycle matrix whose cycle set is isomorphic to $(U_m\times U_n,\circ)$. Hence $A\otimes B\in \cyc_{mn}$
\end{proposition}
\begin{proof}
    Recall from \cref{lem:product-cycle-set}, the operation `$\circ$' on $U_m\times U_n$ is determined by $(a,b)\circ(c,d)=(a\cdot c,b\star d)$, for all $a,c\in U_m$ and $b,d\in U_n$.
    We need to relabel the elements $(i,j)\in U_m\times U_n$, to establish the result.
    Define $\varphi:U_m\times U_n\longrightarrow U_{mn}$ by
    \begin{align*}
        \varphi(i,j)=(i-1)n+j.
    \end{align*}
    Clearly, this function is bijective. Indeed $\varphi(i_1,j_1)=\varphi(i_2,j_2)$ implies that $(j_1-j_2)$ is divisible by $n$, which is not possible. Defining a binary operation `$\bullet$' on
    $U_{mn}$ as
    \begin{align*}
        x\bullet y=\varphi(\varphi^{-1}(x)\circ\varphi^{-1}(y)),
    \end{align*}
    concludes the proof.
\end{proof}

\begin{remark}
Chose any column say $C_j$ of a cycle matrix $(m_{ij})$. Then visit the columns $C_{m_{rj}}$ for all $r=1,2,\cdots,n$ and keep on doing the same process for each $C_{m_{rj}}$. The cycle matrix is indecomposable if and only if all columns can be traversed by this process.
Indeed this is the necessary and sufficient condition for the action of the permutation group $\mathcal{G}$ on $U_n$ to be transitive.
\end{remark}
\begin{corollary}\label{cor:product-indc-imply-indc}
If $(X\times Y,\circ)$ is indecomposable, then both $(X,\cdot)$ and $(X,\star)$ are indecomposable.
\end{corollary}
\begin{proof}
    This is clear from the following inequality
        \begin{align*}
            \orb_{\mathcal{G}(X\times Y)}(x,y)\subset \orb_{\mathcal{G}(X)}(x)\times \orb_{\mathcal{G}(Y)}(y)
        \end{align*}
    for all $(x,y)\in X\times Y$. For the inequality, observe that
    \begin{align*}
        &\orb_{\mathcal{G}(X\times Y)}(x,y)\\
        =&\left\{ \psi^{r_{11}}_{(x_1,y_2)}\psi^{r_{12}}_{(x_1,y_2)}\cdots \psi^{r_{1m}}_{(x_1,y_m)} \psi^{r_{21}}_{(x_2,y_1)}\cdots \psi^{r_{2m}}_{(x_2,y_m)}\cdots \psi^{r_{nm}}_{(x_n,y_m)}(x,y)~\vert~r_{i,j}\in \mathbb{Z}\right\}\\
        =&\left\{\left(\psi^{\sum\limits_{i}^mr_{1i}}_{x_1}\psi^{\sum\limits_{i}^mr_{2i}}_{x_2}\cdots \psi^{\sum\limits_{i}^mr_{ni}}_{x_n}(x),~\psi^{r_{11}}_{y_1}\cdots \psi^{r_{1m}}_{y_m}\cdots \psi^{r_{n1}}_{y_1}\cdots \psi^{r_{nm}}_{y_m}(y)\right)
    \right\}\\
    \subset&\orb_{\mathcal{G}(X)}(x)\times \orb_{\mathcal{G}(Y)}(y)
    \end{align*}
\end{proof}
Note that, the converse of the statement does not hold true. 
We provide the following example exhibiting this scenario. Consider the following two matrices:
\begin{align*}
    \begin{blockarray}{cccc}
& 1 & 2 & 3 \\
\begin{block}{c[ccc]}
  1 & 2 & 3 & 1 \\
  2 & 2 & 3 & 1\\
  3 & 2 & 3 & 1\\
\end{block}
\end{blockarray}, \begin{blockarray}{cccc}
& 1 & 2 & 3 \\
\begin{block}{c[ccc]}
  1 & 3 & 1 & 2 \\
  2 & 3 & 1 & 2\\
  3 & 3 & 1 & 2\\
\end{block}
\end{blockarray}.
\end{align*}
From the tensor matrix, it can be seen that $\orb(1,1)=\{(1,1),~(2,3),~(3,2)\}$, which indicates that the product solution is not an indecomposable solution.

Now we record a special kind of cycle matrix, which may be of independent interest. This has a resemblance with Latin squares as was pointed out by L Vendramin. We start with a definition.
\begin{definition}
    A cycle matrix $M$ will be called a \emph{transpose cycle matrix} if the transpose matrix $M^t$ is also a cycle matrix.
\end{definition}
Note that if an $n\times n$ matrix $M=(m_{ij})$ is a transpose cycle matrix, we get that $(X_n,\cdot)$ is also a cycle set where $i\cdot j= m_{ji}$. 
We exhibit one example here.
\begin{example}
For $n=4$, we have the following two examples which give two non-isomorphic square-free cycle sets. The matrices are:
\begin{align*}
    \begin{bmatrix}
    1 & 4 & 3 & 2 \\
   2 & 3 & 4 & 1\\
   4 & 1 & 2 & 3\\
   3 & 2 & 1 & 4
   \end{bmatrix},
   \begin{bmatrix}
    4 & 2 & 3 & 1 \\
   3 & 1 & 4 & 2\\
   1 & 3 & 2 & 4\\
   2 & 4 & 1 & 3    
   \end{bmatrix}.
\end{align*}
\end{example}
\begin{definition}[Transpose cycle set ]
A cycle set $(X,\cdot)$ is said to be a \emph{transpose cycle set} if $(X,\star)$ is a cycle set, where $x\star y=y\cdot x.$
\end{definition}
\begin{lemma}\label{lem:transpose-cycle}
    A cycle set $(X,\cdot)$ is a transpose cycle set if and only if
    \begin{enumerate}
        \item for all $y\in X$, the map $x\mapsto x\cdot y$ is bijective,
    \item for all $x,~y,~z\in X$ we have$(z\cdot x)\cdot (y\cdot x)=(z\cdot y)\cdot (x\cdot y)$.
    \end{enumerate}
\end{lemma}
The proof of \cref{lem:transpose-cycle} is obvious, from the definition of cycle set. We end this section by mentioning the following lemma, which proves the existence of an infinite family of
transpose cycle sets.
\begin{proposition}\label{prop:infinite-transpose-cycle}
The product of two transposed cycle sets is a transposed cycle set.
\end{proposition}
\begin{proof}
Let $(x',y'),~(x,y)\in X\times Y,$ then there is $(x_1,y_1)\in X\times Y$ such that $x_1\cdot x=x'$ and $(y_1\cdot y=y').$ Therefore $(x_1,y_1)\to (x_1,y_1)\circ (x,y)$ is surjective, hence it is bijective. Let $(x_1,y_1),~(x_2,y_2)$ and $(x_3,y_3)\in X\times Y$. Then we have,
\begin{align*}
&\big(((x_1,y_1)\circ (x_2,y_2)\big)\circ \big(((x_3,y_3)\circ (x_2,y_2)\big)\\
=&(x_1\cdot x_2,y_1\cdot y_2)\circ (x_3\cdot x_2,y_3\cdot y_2)\\
=&\big((x_1\cdot x_2)\cdot (x_3\cdot x_2),~(y_1\cdot y_2)\cdot (y_3\cdot y_2)\big)\\
=&\big((x_1\cdot x_3)\cdot (x_2\cdot x_3),(y_1\cdot y_3)\cdot (y_2\cdot y_3)\big)
\\
=&\big((x_1,y_1)\circ (x_3,y_3)\big)\circ \big((x_2,y_2)\circ (x_3,y_3)\big).
\end{align*}
This finishes the proof.
\end{proof}
\begin{remark}
    All transpose cycle sets are irretractable. Also, each column of a transpose cycle matrix is also a permutation, i.e, $x_i\mapsto \psi_{x_i}(y)$ is a permutation for all $y.$
\end{remark}

%% file: sec_Sn_action.tex
\section{An action of the symmetric group on set of cycle matrices}\label{sec:sn-action}
In this section, we define an action of the symmetric group $\Sym_n$ on $\cyc_n$ and compute the orbit of a given cycle matrix. We start with an example. Consider two isomorphic cycle matrices given by \begin{align*}
 A=
\begin{blockarray}{cccc}
& 1 & 2 & 3 \\
\begin{block}{c[ccc]}
  1 & 2 & 3 & 1 \\
  2 & 2 & 3 & 1\\
  3 & 2 & 3 & 1\\
\end{block}
\end{blockarray}
 \text{ and }B=
\begin{blockarray}{cccc}
& 1 & 2 & 3 \\
\begin{block}{c[ccc]}
  1 & 3 & 1 & 2 \\
  2 & 3 & 1 & 2\\
  3 & 3 & 1 & 2\\
\end{block}
\end{blockarray}.
\end{align*}
Take the element $\sigma=(1,2)\in\Sym_3.$ Transformation of $A$ to $B$ can be observed as follows.
 \begin{align*}
     A\xrightarrow{\sigma(A)}&
\begin{blockarray}{cccc}
& \sigma(1) & \sigma(2) & \sigma(3) \\
\begin{block}{c[ccc]}
  \sigma(1) & \sigma(2) & \sigma(3) & \sigma(1) \\
  \sigma(2) & \sigma(2) & \sigma(3) & \sigma(1)\\
  \sigma(3) & \sigma(2) & \sigma(3) & \sigma(1)\\
\end{block}
\end{blockarray}
= \begin{blockarray}{cccc}
& 2 & 1 & 3 \\
\begin{block}{c[ccc]}
  2 & 1 & 3 & 2 \\
  1 & 1 & 3 & 2\\
  3 & 1 & 3 & 2\\
\end{block}
\end{blockarray}\\
 \xrightarrow{\sigma^{-1}(R)}& \begin{blockarray}{cccc}
& 2 & 1 & 3 \\
\begin{block}{c[ccc]}
  1 & 1 & 3 & 2 \\
  2 & 1 & 3 & 2\\
  3 & 1 & 3 & 2\\
\end{block}
\end{blockarray}
\xrightarrow{\sigma^{-1}(C)} \begin{blockarray}{cccc}
& 1 & 2 & 3 \\
\begin{block}{c[ccc]}
  1 & 3 & 1 & 2 \\
  2 & 3 & 1 & 2\\
  3 & 3 & 1 & 2\\
\end{block}
\end{blockarray}=B,
 \end{align*}
 where $\sigma^{-1}{R}$ means action of $\sigma^{-1}$ on rows, i.e, $R_i\to R_{\sigma^{-1}(i)}$, similarly $\sigma^{-1}(C)$ for columns.
In general, if $A$ and $B$ are two isomorphic solutions and $\sigma~:~A\to B$ is an isomorphism then the action can be seen to be
$A\xrightarrow{\sigma(A)}A_1\xrightarrow{\sigma^{-1}(R)}A_2\xrightarrow{\sigma^{-1}(C)}B$.
Conversely, if $A$ and $B$ are two matrices corresponding to two solutions $(X,\cdot)$ and $(X,\star)$ and there is a permutation $\sigma$ on $X$ such that $B$ can be obtained from $A$ by the above operations, then $(X,\cdot)$ and $(X,\star)$ are isomorphic and $\sigma$ is an isomorphism. The following lemma is the generalized picture. This is crucial for defining the action. The below lemma is essentially the \cite[Lemma 2.5]{AkMeVe22}.
\begin{lemma}\label{lem:action-of-Sn}
    Let $(X,\cdot)$ be a cycle set and $\sigma$ is a permutation on $X$.
    Then $(X,\star)$ is a cycle set, where $x\star y=\sigma(\sigma^{-1}(x)\cdot\sigma^{-1}(y))$.
    Also $\sigma~:~(X,\cdot)\to (X,\star)$ is an isomorphism. 
    Conversely, all cycle sets isomorphic to $(X,\cdot)$ defined on $X$ come like that.
    In particular, $\sigma$ is an automorphism of $(X,\cdot)$ if and only if the cycle matrices of $(X,\cdot)$ and $(X,\star)$ are the same with $X=\{1,2,\cdots,n\}$.
\end{lemma}
\begin{proof}
    It is easy to check that $(X,\star)$ is a cycle set. We have the relation $x\star y=\sigma(\sigma^{-1}(x)\cdot\sigma^{-1}(y)).$ By replacing $x,~y$ by $\sigma(x),~\sigma(y)$ we obtain $\sigma(x)\star\sigma( y)=\sigma(x\cdot y).$ Hence $\sigma$ is an isomorphism between $(X,\cdot)$ and $(X,\star).$
    
    Conversely assume $(X,\star)$ is isomporphic to $(X,\cdot)$ and $\sigma$ be an isomorphism from $(X,\cdot)$ to $(X,\star)$. Then $\sigma(x)\star \sigma(y)=\sigma(x\cdot y).$ By replacing $x,~y$ by $\sigma^{-1}(x),~\sigma^{-1}(y)$, we get our desired result. The last statement is clear.
\end{proof}
\begin{theorem}\label{th:cycle matrices-solution}
There is a bijective correspondence between the set of orbits $\cyc_n/\Sym_n$ and set-theoretic non-degenerate involutive solutions of QYBE of order $n$.    
\end{theorem}
\begin{proof}
    Note that using \cref{lem:action-of-Sn}, we get that there is a bijection between $\cyc_n/\Sym_n$ and the set of all non-isomorphic cycle sets of order $n$. The assertion follows from the bijection between non-isomorphic cycle sets and set-theoretic non-degenerate involutive solutions of QYBE of order $n$.
\end{proof}

Consider the following matrices:
\begin{align*}
    A= \begin{array}{c|cccccc}
    & 1 & 2 & 3 & 4 & 5 & 6 \\
    \hline
    1 & 1 & 2 & 3 & 5 & 4 & 6\\
    2 & 1 & 2 & 6 & 5 & 4 & 3\\
    3 & 4 & 5 & 3 & 1 & 2 & 6\\
    4 & 2 & 1 & 3 & 4 & 5 & 6\\
    5 & 2 & 1 & 6 & 4 & 5 & 3\\
    6 & 4 & 5 & 3 & 1 & 2 & 6
    \end{array},~B =
    \begin{array}{c|cccccc}
      & 3 & 4 & 5 & 2 & 1 & 6 \\
    \hline
    3 & 3 & 4 & 5 & 1 & 2 & 6\\
    4 & 3 & 4 & 6 & 1 & 2 & 5\\
    5 & 2 & 1 & 5 & 3 & 4 & 6\\
    2 & 4 & 3 & 5 & 2 & 1 & 6\\
    1 & 4 & 3 & 6 & 2 & 1 & 5\\
    6 & 2 & 1 & 5 & 3 & 4 & 6
    \end{array},~C =
    \begin{array}{c|cccccc}
      & 1 & 2 & 3 & 4 & 5 & 6 \\
    \hline
    1 & 1 & 2 & 4 & 3 & 6 & 5 \\
    2 & 1 & 2 & 4 & 3 & 5 & 6 \\
    3 & 2 & 1 & 3 & 4 & 5 & 6 \\
    4 & 2 & 1 & 3 & 4 & 6 & 5 \\
    5 & 4 & 3 & 2 & 1 & 5 & 6 \\
    6 & 4 & 3 & 2 & 1 & 5 & 6
    \end{array}.
\end{align*}
Note that $A$ represents a cycle matrix. Considering $\sigma= (1,3,5)(2,4)$ and applying to each entry of $A$, we get the matrix $B$. After rearranging we get the matrix $C$, which is again a cycle matrix. Hence 
for any permutation of $\{1,2,\ldots, n\}$ and a cycle matrix, we get another cycle matrix.
Generalizing this we define the following. 
\begin{definition}[Symmetric group action on the set of cycle matrices]\label{defn:symmetric-group-action} Consider $\M$ to be a cycle matrix and $(X,\star)$ to be the solution associated with it, i.e. $\M_{ij}=i\star j$. For an element $\sigma\in \Sym_n$ define
\begin{align*}
    \left(\sigma\M\right)_{ij}=\sigma(\sigma^{-1}(i)\star\sigma^{-1}(j)).
\end{align*}
\end{definition}
Note that the above definition of action is well-defined because of \cref{lem:action-of-Sn}. Clearly, we have that $e \M=\M$ for the identity element $e\in\Sym_n$.
For two elements $\alpha,\beta\in\Sym_n$, we have that
\begin{align*}
    \left((\alpha\beta)\M\right)_{ij}
    &=\alpha\beta(\beta^{-1}(\alpha^{-1}(i))\star\beta^{-1}(\alpha^{-1}(j)))\\
    \text{and }\left(\alpha\left(\beta\M\right)\right)_{ij}&=\alpha\left(\left(\beta\M\right)_{\alpha^{-1}(i)\alpha^{-1}(j)}\right)\\
    &=\alpha\beta(\beta^{-1}(\alpha^{-1}(i))\star\beta^{-1}(\alpha^{-1}(j))),
\end{align*}
which clearly proves that it is an action. Using \cref{lem:action-of-Sn} and \cref{defn:symmetric-group-action}, the following result is immediate. We note it down here for further reference.
\begin{lemma}\label{lem:orbit-stabilizer-of-action}Let $|X|=n$. Then two cycle matrices (equivalently their solutions) are isomorphic if and only if they lie in the same orbit of the $\Sym_n$ action. 
Also, $\sigma\in\Sym_n$ is an automorphism of $(X,\cdot)$ if and only if it stabilizes the cycle matrix of $(X,\cdot)$.
\end{lemma}
\begin{lemma}\label{lem:cycle-decomposition}
    If two cycle matrices give isomorphic solutions, then they have the same number of rows with the same cycle decomposition.
\end{lemma}
\begin{proof}
    Let $f$ be an isomorphism between the solutions $(X,\cdot)$ and $(X,\star)$. Then, for any $x,y\in X$ 
    \begin{equation*}
        \begin{split}
            f(x\cdot y) &=f(x)\star f(y)\\
            f\psi_x(y) &=\psi'_{f(x)}f(y)\\
            f\psi_x f^{-1}&=\psi'_{f(x)}.\\
        \end{split}
    \end{equation*}
    Hence $\psi_x$ and $\psi'_{f(x)}$ have the same cycle decomposition for all $x\in X.$ The bijectiveness of $f$ completes the proof.
\end{proof}
Note that the converse of the above lemma is not true.
\begin{example}
	\begin{align*}
		A=\begin{array}{c|cccc}
                &1&2&3&4\\
            \hline
			1&1&2&4&3\\
			2&1&2&4&3\\
			3&1&2&4&3\\
			4&1&2&4&3
		\end{array},~
		B=\begin{array}{c|cccc}
		&1&2&3&4\\
            \hline
            1&1&2&4&3\\
		2&1&2&4&3\\
		3&2&1&3&4\\
		4&2&1&3&4
	\end{array}
	\end{align*}
 These are non-isomorphic solutions simply because their permutation groups are non-isomorphic.
\end{example}
\begin{lemma}\label{lem:automorphism-description}
	Let $\psi_i$ denote the $i$-th row of the cycle matrix with the
 corresponding solution $(X,\cdot)$. Then $\alpha\in\mathfrak{S}_n$ is an automorphism of 
	$(X,\cdot)$ if and only if $\alpha\psi_i=\psi_{\alpha(i)}\alpha$.
\end{lemma}
\begin{proof}
    The proof follows the line of the proof for \cref{lem:cycle-decomposition}. We omit it here, constraining ourselves from repetition.
\end{proof}
\begin{corollary}
Let $(X,\cdot_\sigma)$ be a permutation solution corresponding to a permutation $\sigma\in\Sym_n$. Then $\Aut(X,\cdot_{\sigma})\equiv \cent_{\Sym_n}(\sigma)$.
Furthermore $\orb((X,\cdot_\sigma))$ consists of permutation solutions $(X,\cdot_\delta)$ corresponding to $\delta$, where $\delta$ has cycle-structure as same as $\sigma\in\Sym_n$. 
\end{corollary}
\begin{corollary}
There are exactly $\pa(n)$ many permutation solutions of size $n$ up to isomorphism.
\end{corollary}

%% file: sec_construction.tex
\section{Construction of new solutions using cycle-matrices}\label{sec:construction-old-new}
We start this section by fixing notations. This is crucial for understanding the construction of new solutions. For $i=1,2$ consider solutions $(X_i,\cdot_i)$ on the set $X_i=\{1,2,\ldots,k_i\}$. Then we have two cycle matrices $M_1$ and $M_2$. For $\sigma_1,\sigma_2,\ldots,\sigma_{k_2}\in\Sym_{k_1}$ and $\mu_1,\mu_2,\ldots,\mu_{k_1}\in\Sym_{k_2}$, define the matrix $\mathbb{M}$ to be
\begin{align*}
    \mathbb{M}_{ij}=\begin{cases}
        (M_1)_{ij} & \text{if } 1\leq i,j\leq k_1\\
        (M_2)_{i-k_1,j-k_1}+k_1 & \text{if }k_1+1\leq i,j\leq k_1+k_2\\
        \sigma_{i-k_1}(j) & \text{if } k_1+1\leq i\leq k_1+k_2,~1\leq j\leq k_1\\
        \mu_{i}(j-k_1)+k_1 & \text{if } 1\leq i\leq k_1, k_1+1\leq j\leq k_1+k_2
    \end{cases}.
\end{align*}
This matrix will be put as 
\begin{align*}
    \left[\begin{array}{ccc|ccc}
	     &&&&\mu_1&\\
      &\Huge{X_1}&&&\vdots&\\
      &&&&\mu_{k_1}&\\
      \hline
      &\sigma_1&&&&\\
      &\vdots&&&\Huge{X_2}&\\
      &\sigma_{k_2}&&&&
	\end{array}\right].
\end{align*}
If all $\mu_l$'s (respectively $\sigma_k$'s) are same, say $\mu$ (respectively $\sigma$), then this matrix will be simply put as $\left[\begin{array}{c|c}
     X_1&\mu  \\
     \hline
     \sigma&X_2 
\end{array}\right]$.
There will be more variants of this matrix throughout this section. We hope that the matrices will be clear from the context. We exhibit one example to 
provide a better understanding of the above construction.
\begin{example} For example if we take two cycle matrices 
\begin{align*}
    M_1=\left[\begin{array}{ccc}
         1 & 2 & 3 \\
         1 & 2 & 3\\
         1 & 2 & 3
    \end{array}\right],
    M_2=\left[\begin{array}{cc}
         2 & 1 \\
         2 & 1
    \end{array}\right],
\end{align*}
and $\sigma_1=(1,2,3),\sigma_2=(1,2)(3)$, $\mu_1=(1,2),\mu_2=(1,2),\mu_3=(1)(2)$, then the abovementioned matrix will be
\begin{align*}
    \left[\begin{array}{ccccc}
        1 & 2 & 3 & 5 & 4 \\
        1 & 2 & 3 & 5 & 4 \\
        1 & 2 & 3 & 4 & 5 \\
        2 & 3 & 1 & 5 & 4 \\
        2 & 1 & 3 & 5 & 4 
    \end{array}\right].
\end{align*}
\end{example}
\begin{theorem}\label{thm:identity-partition-any}
	Consider $(X_1,\star_1)$ and $(X_2,\star_2)$ to be two trivial solutions of size $k_1$, and $k_2$ respectively. Let a partition on $X_1$ be $^1X_1,~^2X_1,\cdots,~^kX_1$, where $^iX_1=\{r_i,r_i+1,\cdots, r_{i+1}-1\}$ for $i=1,~2,~\cdots, k-1$ and $^kX_1=\{r_k,r_k+1,\cdots,k_1\}$ and $r_1=1.$  Let $^i\alpha_1\in\Sn(r_i,r_i+1,\ldots,r_{i+1}-1),~^i\alpha_2\in\Sn(k_1+1,k_1+2,\ldots,k_1+k_2)$ for $i=1,2,\cdots, k$ such that $^i\alpha_2 ~^j\alpha_2=~^j\alpha_2~ ^i\alpha_2$ for all $i,~j$. Then the matrix 
\begin{align*}
           \mathbb{M}_{(X_i),(\alpha_i),(),[k]} = 
\left(\begin{array}{@{}c|c@{}}
    \bigxone 
    & 
    \begin{array}{c}
         ~^1{\alpha_2} \\ \hline
         \vdots \\ \hline 
         ~^k{\alpha_2} \\
    \end{array}\\
    \hline
    \begin{array}{c|c|c}
        ~^1{\alpha_1} & \cdots & ~^k{\alpha_1} 
    \end{array}
    & \bigxtwo
\end{array}\right)
\end{align*}
is a cycle matrix.
\end{theorem}
\begin{proof}
Define a binary operation $\star$ on $X=X_1\sqcup X_2$ in the following way;
\begin{align*}
&\star|_{X_1\times X_1}=\star_1,
\star|_{X_2\times X_2}=\star_2,\\
&\star|_{~^iX_1\times X_2}(x,x_2)=~^i\alpha_2(x_2), 
\star|_{X_2\times ~^iX_1}(x_2,x)=~^i\alpha_1(x),    
\end{align*}
for all $x\in ~^iX_1$, $x_2\in X_2$ and $i=1,2,\ldots,k$.
We aim to show that $(X,\star)$ is a cycle set.
It is easy to see from the definition of $\star$ that each row and principal diagonal of the matrix $ \mathbb{M}_{(X_i),(\alpha_i),(),[k]}$ is a permutation.
Now we show the cycloid relation. We do it by considering various cases.

\textbf{Case 1:} If all of $x,y,z$ lies in one of $X_1$ or $X_2$, the cycloid relation trivially holds.

\textbf{Case 2: } $x\in ~^iX_1$, $y\in ~^jX_1$ and $z\in X_2$
\begin{equation*}
	\begin{split}
		(x\star y)\star (x \star z)&=y\star ~^i\alpha_2(z)=~^j\alpha_2~^i\alpha_2(z),\\
		(y\star x)\star (y \star z)&=x\star ~^j\alpha_2(z)=~^i\alpha_2~^j\alpha_2(z).
	\end{split}
\end{equation*}

\textbf{Case 3: } $x\in ~^iX_1$, $y\in ~X_2$ and $z\in ~^jX_1$
\begin{equation*}
	\begin{split}
		(x\star y)\star (x \star z)&=~^i\alpha_2
		(y)\star z=~^j\alpha_1(z),\\
		(y\star x)\star (y \star z)&=~^i\alpha_1(x)\star ~^j\alpha_1(z)=~^j\alpha_1(z).
	\end{split}
\end{equation*}

\textbf{Case 4: } $x\in ~^iX_1$, $y\in ~X_2$ and $z\in X_2$
\begin{equation*}
	\begin{split}
		(x\star y)\star (x \star z)&=~^i\alpha_2
		(y)\star ~^i\alpha_2(z)=~^i\alpha_2(z),\\
		(y\star x)\star (y \star z)&=~^i\alpha_1(x)\star z=~^i\alpha_2(z).
	\end{split}
\end{equation*}

\textbf{Case 5: } $x\in X_2$, $y\in ~^iX_1$ and $z\in ~^jX_1$
\begin{equation*}
	\begin{split}
		(x\star y)\star (x \star z)&=~^i\alpha_1
		(y)\star ~^j\alpha_1(z)=~^j\alpha_1(z),\\
		(y\star x)\star (y \star z)&=~^i\alpha_2(x)\star z=~^j\alpha_2(z).
	\end{split}
\end{equation*}

\textbf{Case 6: } $x\in X_2$, $y\in ~^iX_1$ and $z\in X_2$
\begin{equation*}
	\begin{split}
		(x\star y)\star (x \star z)&=~^i\alpha_1
		(y)\star z=~^i\alpha_2(z),\\
		(y\star x)\star (y \star z)&=~^i\alpha_2(x)\star ~^i\alpha_2(z)=~^i\alpha_2(z).
	\end{split}
\end{equation*}

\textbf{Case 7: } $x\in X_2$, $y\in ~X_2$ and $z\in ~^iX_1$
\begin{equation*}
	\begin{split}
		(x\star y)\star (x \star z)&=y\star ~^i\alpha_1(z)=~^i\alpha_1~^i\alpha_1(z),\\
		(y\star x)\star (y \star z)&=x\star ~^i\alpha_1(z)=~^i\alpha_1~^i\alpha_1(z).
	\end{split}
\end{equation*}
Therefore $(X,\star)$ is a cycle set and, the cycle matrix is given by the above matrix.
\end{proof}
Continue with the notation of $\cref{thm:identity-partition-any}.$ 
\begin{corollary}\label{cor:multipermutation 2}
The constructed cycle-matrix $\mathbb{M}_{(X_i),(\alpha_i),(),[k]}$, is a multipermutation solution of level $2.$
\end{corollary}
\begin{proof}
    $Ret(\mathbb{M}_{(X_i),(\alpha_i),(),[k]})$ is a cycle-matrix of a trivial solution.
\end{proof}
\begin{corollary}
    Every finite abelian group is a permutation group.
\end{corollary}
\begin{proof}
    Let $H=\langle \sigma_1,\sigma_2,\cdots,\sigma_r\rangle$ to be a finite abelian group.
    Then take $X_1$ to be of size $1$, $X_2$ of size $r$ and $^i\alpha_2=\sigma_i$. Then applying \cref{thm:identity-partition-any} we have the result.
\end{proof}

\begin{proposition}\label{prop:2-times-2}
Consider $(X_1,\star_1)$ and $(X_2,\star_2)$ to be two solutions of size $k_1$, and $k_2$ respectively. Let $\alpha_i\in\Aut(X_i,\star_i)$, for $i=1,2$. Then the matrix 
\begin{align*}
    \mathbb{M}_{(X_i),(\alpha_i),(12),[~]} =\left[\begin{array}{c|c}
    X_1&\alpha_2
    \\\hline
    \alpha_1
&X_2
    \end{array}\right]
\end{align*}
is a cycle matrix.
\end{proposition}
\begin{proof}
    Define the following binary operation `$\star$' on $X=X_1\sqcup X_2$ in the following way:
    \begin{align*}
      \star|_{X_1\times X_1}=\star_1, & \star|_{X_1\times X_2}(x_1,x_2)=\alpha_2(x_2),\\
	\star|_{X_2\times X_1}(x_2,x_1)=\alpha_1(x_1), & \star|_{X_2\times X_2}=\star_2,   
    \end{align*}
	 for all $x_i\in X_i$, $i=1,2$. We show that $(X,\star)$ is a cycle set.
	Assuming $x\in X$, from the definition of $\star$ we get that,
 \begin{align*}
     &^X\psi_x(x_1)=^{X_1}\psi_x(x_1),~^X\psi_x(x_2)=\alpha_2(x_2),\\
	&^X\psi_x(x_2)=^{x_2}\psi_x(x_2),~^X\psi_x(x_1)=\alpha_1(x_1),
 \end{align*} for $x\in X_2,~x_1\in X_1$ and $x_2\in X_2.$ This proves that $^X\psi_x$ is a permutation on $X.$ Next let us denote the diagonal maps for $X_1$ and $X_2$ by $\pi_1$ and $\pi_2$ respectively. Then the map $x\mapsto x\star x$ on $X$ will be of the form $x\mapsto \pi_1(x)$ if $x\in X_1$ and $x\mapsto \pi_2(x)$ for $x\in X_2$. Hence the diagonal map $x\mapsto x\star x$ is bijective on $X.$ We finally show that the cycloid relation holds.
	The proof will be divided into several cases.
	Let $x,~y$ and $z\in X$. We start with the cases where $x\in X_1$.
 
\textbf{Case 1: ($y,~z\in X_1$)}
	In this case, the cycloid relation trivially holds.
 
\textbf{Case 2: ($y\in X_1$, $z\in X_2$)} 
	\begin{equation*}
		\begin{split}
			(x\star y)\star (x \star z)&=( ^{X_1}\psi_x(y))\star \alpha_2(z)=\alpha_2(\alpha_2(z)),\\
			(y\star x)\star (y \star z)&=( ^{X_1}\psi_y(x))\star \alpha_2(z)=\alpha_2(\alpha_2(z)).
		\end{split}
	\end{equation*}

\textbf{Case 3:  ($z\in X_1$, $y\in X_2$)}
    	\begin{equation*}
    	\begin{split}
    		(x\star y)\star (x \star z)&=\alpha_2(y)\star ^{X_1}\psi_x(z)=\alpha_1( ^{X_1}\psi_x(z)),\\
    		(y\star x)\star (y \star z)&=\alpha_1(x)\star \alpha_1(z)= ^{X_1}\psi_{\alpha_1(x)}(\alpha_1(z)).\\
    	\end{split}
    \end{equation*}

\textbf{Case 4: ($y,~z\in X_2$)}
	\begin{equation*}
		\begin{split}
		   (x\star y)\star (x \star z)&=\alpha_2(y)\star\alpha_2(z)=^{X_2}\psi_{\alpha_2(y)}(\alpha_2(z)),\\
		   (y\star x)\star (y \star z)&=\alpha_1(x)\star ^{X_2}\psi_y(z)=\alpha_2(^{X_2}\psi_y(z)).\\
		\end{split}
	\end{equation*}
 Similarly, we can prove the other equalities. Therefore $(X,\star)$ is a cycle set. The cycle matrix of is clearly seen to be $\mathbb{M}_{(X_i),(\alpha_i),(12)}.$
 \end{proof}
 The above solution will be denoted by $X_1\bigcup\limits_{\alpha_1,\alpha_2}X_2$, hereafter.
\begin{example}
Consider the solutions $(X_1,\star_1)$ and $(X_2,\star_2)$ with cardinality $2$ and $3$ respectively. Further, assume the cycle matrix to be of the form
\begin{align*}
    \begin{bmatrix}
    1&2\\
    1&2
    \end{bmatrix}, \begin{bmatrix}
        1&2&3\\
        1&2&3\\
        1&2&3
    \end{bmatrix}.
\end{align*}
Then taking $\alpha_1=(1,2)$ and $\alpha_2=(1,2,3)$ we get a new solution of cardinality $5$ given by cycle-matrix 
\begin{align*}
    \begin{bmatrix}
        1&2&4&5&3\\
        1&2&4&5&3\\
        2&1&3&4&5\\
        2&1&3&4&5\\
        2&1&3&4&5
    \end{bmatrix}.
\end{align*}
\end{example}

\begin{theorem}\label{thm:2-times-2-general}
Let $X_1,X_2,\cdots, X_l$ be $l$ solutions of size $k_1,k_2,\cdots,k_l$ respectively. Let $\alpha_i\in\Aut(X_i,\cdot_i)$ for $1\leq i\leq l$. Further, consider 
\begin{align*}
\alpha_{1,2,\ldots,t} \in \Aut \left(
    \left(\ldots\left(X_1\bigcup\limits_{\alpha_1,\alpha_2}X_2\right)\ldots\bigcup\limits_{\alpha_{1,2,\ldots,t-1},\alpha_t}X_t
    \right)
\right)
\end{align*}
for $2\leq t\leq l-1$. Then the matrix
\begin{align*}
,    \left[
        \begin{array}{c|c}
            \begin{array}{cc}
            \begin{array}{c|c|}
                \begin{array}{c|c}
                     \bigxone  & \alpha_2 \\ \hline
                     \alpha_1 & \bigxtwo
                 \end{array} & \alpha_3 \\ \hline
                 \alpha_{1,2} & \bigxthree \\ 
                 \hline
            \end{array} 
            & 
            \begin{array}{ccc}
                 &  & \\
                 &  & \\
                 &  & 
            \end{array} 
            \\
            \begin{array}{ccc}
                 &  & \\
                 &  & \\
                 &  & 
            \end{array}
            & 
            \ddots 
        \end{array} 
        & 
        \biga_l
        \\
        \hline
        \biga_{1,2,\ldots,l-1} 
        & 
        X_l
        \end{array}
    \right],
\end{align*}
is a cycle matrix.
\end{theorem}
\begin{proof}
    The proof is by induction. Note that the case $n=2$ is the result of \cref{prop:2-times-2}. Now assume the result to be true for $l=n$. 

    Assume $\beta=\alpha_{1,2,\ldots,n+1}$ is an automorphism of $\widetilde{X}=\left(\ldots\left(X_1\bigcup\limits_{\alpha_1,\alpha_2}X_2\right)\ldots\bigcup\limits_{\alpha_{1,2,\ldots,n},\alpha_n}X_n\right)$. Given that $\alpha_l\in\Aut(X_l,\cdot_l)$, using \cref{prop:2-times-2} we get that
    \begin{align*}
        \left[\begin{array}{c|c}
            \widetilde{X} & \alpha_l \\\hline
             \beta & X_l
        \end{array}\right]
    \end{align*}
    is a cycle matrix. Note that this matrix is as same as the matrix given in the statement. This finishes the proof.
\end{proof}
\begin{proposition}\label{prop:3-times-3}
    Consider $X_1$, $X_2$, and $X_3$ to be three solutions of size $k_1,~k_2$ and $k_3$ respectively. Let $\alpha_i\in\Aut(X_i,\cdot_i)$ for $i=1,2,3$. Then the following two matrices
    \begin{align*}
        \mathbb{M}_{(X_i),(\alpha_i),(123),[~]}=\left[\begin{array}{c|c|c}
            X_1 &  & \alpha_3\\
             \hline
            \alpha_1 & X_2 &\\
             \hline
             & \alpha_2 & X_3
        \end{array}\right] ~\text{and}~
        \mathbb{M}_{(X_i),(\alpha_i),(12),[~]}=\left[\begin{array}{c|c|c}
            X_1 & \alpha_2 & \\
             \hline
             & X_2 & \\
             \hline
            \alpha_1 &  & X_3
        \end{array}\right]
    \end{align*}
    are two cycle matrices.
\end{proposition}
\begin{proof}
	Define $\star$ on $X=X_1\sqcup X_2\sqcup X_3$ as follows
 \begin{align*}
    &\star|_{X_i\times X_i}=\star_i,\\
    &\star|_{X_2\times X_3}(x_2,x_3)=x_3, 
    \star|_{X_3\times X_1}(x_3,x_1)=x_1,\star|_{X_1\times X_2}(x_1,x_2)=x_2,\\
	&\star|_{X_1\times X_3}(x_1,x_3)=\alpha_3(x_3),
 \star|_{X_2\times X_1}(x_2,x_1)=\alpha_1(x_1),\star|_{X_3\times X_2}(x_3,x_2)=\alpha_2(x_2),
 \end{align*}
	for all $x_i\in X_i$, $i=1,2,3.$ It is easy to see from the matrix $\mathbb{M}_{(X_i),(\alpha_i),(123),[~]}$ that all rows and the principal diagonal are permutations. So we only need to show the cycloid relation.
    From \cref{prop:2-times-2} observe that $\mathbb{M}_{(X_1,X_2),(\alpha_1, id),(1,2),[~]}$, and $\mathbb{M}_{(X_2,X_3),(\alpha_2, id),(2,3),[~]}$ are cycle-matrices. Also, the matrix 
    \begin{align*}
    \left[\begin{array}{c|c}
         X_1&  \Delta_3 \\
         \hline
         \alpha_1& X_3
    \end{array}\right]    
    \end{align*}
    for $\Delta_3\in\{I,\alpha_3\}$, is a cycle matrix.
    Therefore  $\mathbb{M}_{(X_i),(\alpha_i),(123),[~]}$ (for $\Delta_3=\alpha_3$) and $\mathbb{M}_{(X_i),(\alpha_i),(12),[~]}$ (for $\Delta_3=I$) are a cycle-matrices.
\end{proof}

Note that it is possible to construct two non-isomorphic cycle matrices of the form $\mathbb{M}_{(X_i),(\alpha_i),\theta_1,[~]}$, $\mathbb{M}_{(X_i),(\alpha_i),\theta_2,[~]}$ even if $\theta_1$ and $\theta_2$ have the same cycle decomposition. Here is an example.
\begin{example}
Consider
   \begin{align*}
    A= \begin{array}{c|cccc|ccc|cc}
    & 1 & 2 & 3 & 4 & 5 & 6 & 7 & 8 & 9\\
    \hline
    1 & 1 & 2 & 3 & 4 & 6 & 7 & 5 & 8 & 9\\
    2 & 1 & 2 & 3 & 4 & 6 & 7 & 5 & 8 & 9\\
    3 & 1 & 2 & 3 & 4 & 6 & 7 & 5 & 8 & 9\\
    4 & 1 & 2 & 3 & 4 & 6 & 7 & 5 & 8 & 9\\\hline
    5 & 1 & 2 & 3 & 4 & 5 & 6 & 7 & 9 & 8\\
    6 & 1 & 2 & 3 & 4 & 5 & 6 & 7 & 9 & 8\\
    7 & 1 & 2 & 3 & 4 & 5 & 6 & 7 & 9 & 8\\\hline
    8 & 2 & 1 & 4 & 3 & 5 & 6 & 7 & 8 & 9\\
    9 & 2 & 1 & 4 & 3 & 5 & 6 & 7 & 8 & 9
    \end{array},~B =
    \begin{array}{c|cccc|ccc|cc}
    & 1 & 2 & 3 & 4 & 5 & 6 & 7 & 8 & 9\\
    \hline
    1 & 1 & 2 & 3 & 4 & 5 & 6 & 7 & 9 & 8\\
    2 & 1 & 2 & 3 & 4 & 5 & 6 & 7 & 9 & 8\\
    3 & 1 & 2 & 3 & 4 & 5 & 6 & 7 & 9 & 8\\
    4 & 1 & 2 & 3 & 4 & 5 & 6 & 7 & 9 & 8\\\hline
    5 & 2 & 1 & 4 & 3 & 5 & 6 & 7 & 8 & 9\\
    6 & 2 & 1 & 4 & 3 & 5 & 6 & 7 & 8 & 9\\
    7 & 2 & 1 & 4 & 3 & 5 & 6 & 7 & 8 & 9\\\hline
    8 & 1 & 2 & 3 & 4 & 6 & 7 & 5 & 8 & 9\\
    9 & 1 & 2 & 3 & 4 & 6 & 7 & 5 & 8 & 9
    \end{array}.
\end{align*}
Here $X_1,~X_2$ and $X_3$ are trivial solutions of order $4,~3$ and $2$ respectively, $\alpha$ permutations are $\alpha_1=(1,2)(3,4),~\alpha_2=(5,6,7)$ and $\alpha_3=(8,9)$, and the $\theta$ permutations are $\theta_1=(1,3,2),~\theta_2=(1,2,3)$ respectively. From \cref{lem:cycle-decomposition}, these two cycle-matrices are non-isomorphic.  
\end{example}
Now we generalize the previous results and present the general scenario in the last result of this section.
\begin{theorem}\label{thm:3-times-3-general}
    Let $N$ be a positive integer and $\Lambda=(m_1,m_2,\ldots,m_k)$ be a partition of $N$. Consider $(X_i,\star_i)$ to be a solution of size $m_i$ and $\alpha_i\in\Aut(X_i,\star_i)$ for all $i=1,2,\ldots,k$. For an element $\Theta\in S_k$  consider the matrix with $\mu\nu$-th block defined as,
    \begin{align*}
        \left(\mathbb{M}_{(X_i),(\alpha_i),\Theta,[~]}\right)_{\mu\nu}=\begin{cases}
        X_\mu&\text{if } \mu=\nu\\
        \alpha_\nu&\text{if }\Theta(\mu)=\nu\neq\mu\\
        \text{I}&\text{otherwise}
        \end{cases}.
    \end{align*}
 Then the matrix $\mathbb{M}_{(X_i),(\alpha_i),\Theta,[~]}$ is a cycle-matrix.
\end{theorem}
\begin{proof}
    The proof is based on the induction argument and the method of \cref{prop:3-times-3}.
    We start with the first step when $k=2$. 
    Then this is the result \cref{prop:2-times-2}. Now assume the result to be true for $k=n$. 
    For the case, $k=n+1$, note that first considering the matrix for $1\leq \mu,\nu \leq n$, we obtain a cycle matrix.
    Similarly for $2\leq \mu,\nu\leq n+1$ We get another cycle matrix. Now
    \begin{align*}
        \left[\begin{array}{c|c}
             X_1&  \Delta_{n+1}\\\hline
             \Delta^\prime_1& X_{n+1}
        \end{array}\right],
    \end{align*}
    is a cycle matrix for $\Delta_{n+1}\in\left\{id_{X_{n+1}},\alpha_{n+1}\right\}$, $\Delta^{\prime}_{1}\in\left\{id_{X_1},\alpha_{n+1}\right\}$ (using \cref{prop:2-times-2}). Hence it follows that $\mathbb{M}_{(X_i),(\alpha_i),\Theta,[~]}$ is a cycle matrix.
\end{proof}

%% file: sec_multiperm.tex
\section{Construction of multi-permutation solutions}\label{sec:multiperm-solutions}
This section is devoted to the construction of multipermuatation solutions of level $r$ for any $r\geq 1$. These have been studied in the work of Gateva-Ivanova, Cameron (\cite{GICa12}), Jedli\v{c}ka, Pilitowska, and Zamojska-Dzienio (\cite{JePiZa20}) and finally a classification result for second level solutions was obtained recently in \cite{Rum22} by Rump. They are important because, a finite solution $(X,r)$ has the structure group $\mathcal{G}(X,r)$ to be poly-$\mathbb{Z}$ if and only if it is a multipermutation solution (see \cite{JeOk05} and \cite{BaCeVe18}).   
\begin{proposition}\label{prop:multiperm-2}
Consider $X_2$ to be the identity solution and $|X_2|=2$. Construct
\begin{align*}
    X_{2^2}=\left[\begin{array}{c|c}
         X_2&\sigma_2  \\
         \hline
         \sigma_2&X_2
    \end{array}\right],
\end{align*}
where $\sigma_2=(1,2)$. This gives a multipermutation solution of level $2$.
\end{proposition}
\begin{proof}
    From \cref{lem:automorphism-description}, $\sigma_2$ is an automorphism of $X_2.$ Now it follows from \cref{cor:multipermutation 2} that $X_{2^2}$ is a multipermutaion solution of level $2$. 
\end{proof}
\begin{theorem}\label{thm:multiperm-m}
For all $m\geq 1$, with notation as of \cref{prop:multiperm-2} the matrix defined as
\begin{align*}
    X_{2^{m+1}}=\left[\begin{array}{c|c}
         X_{2^{m}}&\sigma_{2^m}  \\
         \hline
         \sigma_{2^m}&X_{2^{m}}
    \end{array}\right],
\end{align*}
with $\sigma_{2^m}=\prod\limits_{i=1}^{2^{m-1}}(i,i+2^{m-1})$ gives a multipermutation level $m+1$ solution of YBE.
\end{theorem}
\begin{proof}
It is easy to see for $m=1.$ For $m=2;$
\begin{align*}
    X_{2^{3}}=&\left[\begin{array}{cccc|cccc}
         1&2&4&3&7&8&5&6  \\
         1&2&4&3&7&8&5&6\\
         2&1&3&4&7&8&5&6\\
         2&1&3&4&7&8&5&6\\
         \hline
         3&4&1&2&5&6&8&7\\
         3&4&1&2&5&6&8&7\\
         3&4&1&2&6&5&7&8\\
         3&4&1&2&6&5&7&8
    \end{array}\right]\xrightarrow{Ret(X_{2^3})}
    \left[\begin{array}{cc|cc}
         \bar{1}&\bar{2}&\bar{4}&\bar{3} \\
         \bar{1}&\bar{2}&\bar{4}&\bar{3}\\
         \hline
         \bar{2}&\bar{1}&\bar{3}&\bar{4}\\
         \bar{2}&\bar{1}&\bar{3}&\bar{4}
    \end{array}\right]=X_{2^{2}}\\\xrightarrow{Ret^2(X_{2^3})}
        &\left[\begin{array}{cc}
         \bar{\bar{1}}&\bar{\bar{2}}\\
         \bar{\bar{1}}&\bar{\bar{2}}
    \end{array}\right]\xrightarrow{Ret^3(X_{2^3})}
    [\Bar{\Bar{\Bar{1}}}],
\end{align*}
where $\Bar{1}=\{1,2\},\Bar{2}=\{3,4\},\Bar{3}=\{5,6\},\Bar{4}=\{7,8\},\bar{\bar{1}}=\{\Bar{1},\Bar{2}\},\bar{\Bar{2}}=\{\Bar{3},\Bar{4}\}~and~\Bar{\Bar{1}}=\{\Bar{\Bar{\bar{1}}},\Bar{\Bar{2}}\}.$ From \cref{prop:multiperm-2}, \cref{lem:automorphism-description} and \cref{prop:2-times-2}, it follows that $X_{2^3}$ is cycle-matrix. Therefore $X_{2^3}$ is a multipermuation level $3$ cycle-matrix.\\
It is clear that in $X_{2^{m+1}},$ $\psi_{2i-1}=\psi_{2i}$  for $i=1,2,\cdots,2^{m}$ and $\psi_{2i}\neq \psi_{2i+1}$ for $i=1,2,\cdots,2^m-1$, therefore $\Bar{i}=\{2i-1,2i\}$ for $i=1,2,\cdots,2^{m}$. Notice that $\psi_i(2j-1),~\psi_i(2j)$ lie in the same class. The matrix $X_{2^{m+1}}$ has four blocks say $B_{ij}$ and each of order $2^{m}\times 2^{m}$. Each row in $B_{12}$ and in $B_{21}$ is, $$[2^{m}+2^{m-1}+1,2^{m}+2^{m-1}+2,\cdots,2^{m+1},2^{m}+1,2^m+2,\cdots, 2^{m}+2^{m-1}]~~and$$ $$[2^{m-1}+1,2^{m-1}+2,\cdots,2^{m},1,2,\cdots, 2^{m-1}]$$ respectively. With this observation and the fact that $\psi_{\Bar{i}\cdot\Bar{j}}=\psi_{\overline{i.j}}$ gives,
$$Ret(X_{2^{m+1}})=X_{2^m}.$$ Now the proof follows from induction.
\end{proof}
\subsection*{Acknowledgement} We thank Prof. Arvind Ayyer (IISc, Bangalore) and, Prof. Manoj Kumar Yadav (HRI, Prayagraj) for their interest in this work. In a personal communication, it was conveyed by Prof. Leonardo Vendramin (Vrije Universiteit Brussel) that similar matrices were used by him to calculate the number of SYBEs for $n\leq 10$ in \cite{AkMeVe22}. We thank Prof L. Vendramin for pointing out the paper \cite{AkMeVe22} to us.

%% file: main.bbl
\begin{thebibliography}{10}

\bibitem{AgChMi23}
{\sc A.~L. Agore, A.~Chirvasitu, and G.~Militaru}, {\em The set-theoretic
  yang-baxter equation, kimura semigroups and functional graphs}, 2023.

\bibitem{AkMeVe22}
{\sc O.~Akg\"{u}n, M.~Mereb, and L.~Vendramin}, {\em Enumeration of
  set-theoretic solutions to the {Y}ang-{B}axter equation}, Math. Comp., 91
  (2022), pp.~1469--1481.

\bibitem{Jes16}
{\sc D.~Bachiller, F.~Ced\'{o}, and E.~Jespers}, {\em Solutions of the
  {Y}ang-{B}axter equation associated with a left brace}, J. Algebra, 463
  (2016), pp.~80--102.

\bibitem{BaCeVe18}
{\sc D.~Bachiller, F.~Ced\'{o}, and L.~Vendramin}, {\em A characterization of
  finite multipermutation solutions of the {Y}ang-{B}axter equation}, Publ.
  Mat., 62 (2018), pp.~641--649.

\bibitem{CeJeDe10}
{\sc F.~Ced\'{o}, E.~Jespers, and A.~del R\'{i}o}, {\em Involutive
  {Y}ang-{B}axter groups}, Trans. Amer. Math. Soc., 362 (2010), pp.~2541--2558.

\bibitem{Dr90}
{\sc V.~G. Drinfeld}, {\em On some unsolved problems in quantum group theory},
  in Quantum groups ({L}eningrad, 1990), vol.~1510 of Lecture Notes in Math.,
  Springer, Berlin, 1992, pp.~1--8.

\bibitem{EtScSo99}
{\sc P.~Etingof, T.~Schedler, and A.~Soloviev}, {\em Set-theoretical solutions
  to the quantum {Y}ang-{B}axter equation}, Duke Math. J., 100 (1999),
  pp.~169--209.

\bibitem{GICa12}
{\sc T.~Gateva-Ivanova and P.~Cameron}, {\em Multipermutation solutions of the
  {Y}ang-{B}axter equation}, Comm. Math. Phys., 309 (2012), pp.~583--621.

\bibitem{JePiZa20}
{\sc P.~Jedli\v{c}ka, A.~Pilitowska, and A.~Zamojska-Dzienio}, {\em The
  construction of multipermutation solutions of the {Y}ang-{B}axter equation of
  level 2}, J. Combin. Theory Ser. A, 176 (2020), pp.~105295, 35.

\bibitem{JeOk05}
{\sc E.~Jespers and J.~Okni\'{n}ski}, {\em Monoids and groups of {$I$}-type},
  Algebr. Represent. Theory, 8 (2005), pp.~709--729.

\bibitem{LuYaZh00}
{\sc J.-H. Lu, M.~Yan, and Y.-C. Zhu}, {\em On the set-theoretical
  {Y}ang-{B}axter equation}, Duke Math. J., 104 (2000), pp.~1--18.

\bibitem{RaVe22}
{\sc S.~Ram\'irez and L.~Vendramin}, {\em Decomposition theorems for involutive
  solutions to the {Y}ang-{B}axter equation}, Int. Math. Res. Not. IMRN,
  (2022), pp.~18078--18091.

\bibitem{Ru05}
{\sc W.~Rump}, {\em A decomposition theorem for square-free unitary solutions
  of the quantum {Y}ang-{B}axter equation}, Adv. Math., 193 (2005), pp.~40--55.

\bibitem{Ru07}
\leavevmode\vrule height 2pt depth -1.6pt width 23pt, {\em Braces, radical
  rings, and the quantum {Y}ang-{B}axter equation}, J. Algebra, 307 (2007),
  pp.~153--170.

\bibitem{Rum22}
\leavevmode\vrule height 2pt depth -1.6pt width 23pt, {\em Classification of
  non-degenerate involutive set-theoretic solutions to the {Y}ang-{B}axter
  equation with multipermutation level two}, Algebr. Represent. Theory, 25
  (2022), pp.~1293--1307.

\bibitem{WeXu}
{\sc A.~Weinstein and P.~Xu}, {\em Classical solutions of the quantum
  {Y}ang-{B}axter equation}, Comm. Math. Phys., 148 (1992), pp.~309--343.

\end{thebibliography}
